\documentclass[11pt]{amsart}
\usepackage{amsmath,amsthm,amsfonts,amssymb,mathrsfs}

\usepackage[all]{xy}

\newcommand*{\vv}[1]{\vec{\mkern1mu#1}\!}

\newcommand{\Fin}{\mathcal{F}\kern-1pt\mathit{in}}
\newcommand{\w}{\omega}
\newcommand{\IN}{\mathbb N}
\newcommand{\IR}{\mathbb R}
\newcommand{\II}{\mathbb I}

\newcommand{\F}{\mathcal F}

\newcommand{\U}{\mathcal U}
\newcommand{\V}{\mathcal V}

\newcommand{\Law}{w^\circ}
\newcommand{\Zar}{w^\bullet}
\newcommand{\Weak}{w}
\newcommand{\Lawson}{\mathcal L}
\newcommand{\Scott}{\mathcal{S}}
\newcommand{\Interval}{\mathcal I}

\newcommand{\DLaw}{\mathcal W^\circ}
\newcommand{\DZar}{\mathcal W^\bullet}
\newcommand{\DWeak}{\mathcal W}
\newcommand{\Chain}{\mathcal C}
\newcommand{\Dedekind}{\mathcal D}

\newcommand{\Tau}{\tau}
\newcommand{\pr}{\mathrm{pr}}
\newcommand{\Ra}{\Rightarrow}

\newtheorem{theorem}{Theorem}[section]
\newtheorem{proposition}[theorem]{Proposition}
\newtheorem{theorem*}{Theorem}
\newtheorem{lemma}[theorem]{Lemma}
\newtheorem{corollary}[theorem]{Corollary}

\theoremstyle{definition}

\newtheorem{remark}[theorem]{Remark}
\newtheorem{problem}[theorem]{Problem}
\newtheorem{example}[theorem]{Example}

\title{The interplay between weak topologies on topological semilattices}
\author{Taras Banakh and Serhii Bardyla}
\dedicatory{Dedicated to the memory of W.W. Comfort}
\address{T.Banakh: Ivan Franko National University of Lviv (Ukraine) and Jan Kochanowski University in Kielce (Poland)}
\email{t.o.banakh@gmail.com}
\address{S.Bardyla: Ivan Franko National University of Lviv (Ukraine)}
\email{sbardyla@yahoo.com}

\begin{document}

\begin{abstract} We study the interplay between three weak topologies on a topological semilattice $X$: the weak$^\circ$ topology $\Law_X$ (generated by the base consiting of open subsemilattices of $X$), the weak$^\bullet$ topology $\Zar_X$ (generated by the subbase consisting of complements to closed subsemilattices), and the $\II$-weak topology $\Weak_X$ (which is the weakest topology in which all continuous homomorphisms $h:X\to [0,1]$ remain continuous).   Also we study the interplay between the weak topologies $\Zar_X$, $\Law_X$, $\Weak_X$ of a topological semilattice $X$ and some intrinsic topologies, determined by the order structure of the semilattice.

We prove that the weak$^\bullet$ topology $\Zar_X$ on a Hausdorff semitopological semilattice $X$ is compact if and only if $X$ is chain-compact in the sense that each closed chain in $X$ is compact.
For a compact Hausdorff topological semilattice $X$ with topology $\Tau_X$ we prove that $\Tau_X=\Weak_X$ iff $\Tau_X=\Zar_X$ iff $\Tau_X=\Law_X$.
 \end{abstract}
\maketitle

\section{Introduction}

A {\em semigroup} is a set $X$ endowed with an associative binary operation $\cdot\colon X\times X\to X$, $\cdot\colon(x,y)\mapsto xy$. A semigroup $X$ is a {\em band} if each element $x\in X$ is an {\em idempotent} in the sense that $xx=x$. A commutative band $X$ is called a {\em semilattice}.  Each band $X$ carries the natural partial order $\le$ defined by $x\le y$ if $xy=x=yx$. A band  endowed with the natural partial order is a {\em poset} (i.e., a {\em partially ordered set}).

For a point $x$ of a poset $(X,\le)$ let ${\uparrow}x:=\{y\in X:x\le y\}$ and ${\downarrow}x:=\{y\in X:y\le x\}$ be the {\em upper} and {\em lower set} of $x$ in $(X,\le)$.
A subset $C$ of a poset $X$ is called a {\em chain} if $x\in {\uparrow}y\cup{\downarrow}y$ for any points $x,y\in C$.



A {\em topologized semigroup} is a semigroup $X$ endowed with a topology $\Tau_X$. A topologized semigroup $X$ is called a
\begin{itemize}
\item a {\em topological semigroup} if the binary operation $X\times X\to X$, $(x,y)\mapsto xy$, is continuous;
\item a {\em semitopological semigroup} if the binary operation $X\times X\to X$, $(x,y)\mapsto xy$, is separately continuous;
\item a {\em subtopological semigroup} if for any subsemigroup $S\subset X$ its closure $\bar S$ in $X$ is a subsemigroup of $X$.
\end{itemize}
It is easy to see that each topological semigroup is semitopological and any semitopological semigroup is subtopological. More information on topological semigroups can be found in the survey \cite{CHR} of Comfort, Hofmann and Remus.

In this paper we shall be mainly interested in (semi)topological semilattices.
An important example of a topological semilattice is the closed interval $\II=[0,1]$ endowed with the semilattice operation $\II\times\II\to\II$, $(x,y)\mapsto \min(x,y)$, of taking minimum.

On each topologized semigroup $X$ we shall consider three weaker topologies:
\begin{itemize}
\item[-] the {\em weak$^{\circ}$ topology} $\Law_X$, generated by the base consisting of open subsemigroups of $X$,
\item[-] the {\em weak$^{\bullet}$ topology} $\Zar_X$, generated by the subbase consisting of complements to closed subsemigroups of $X$,
\item[-] the {\em $\II$-weak topology} $\Weak_X$, generated by the subbase consisting of the preimages $h^{-1}(U)$ of open sets $U\subset\II$ under continuous semigroup homomorphisms $h:X\to\II$.
\end{itemize}

It is clear that
$$
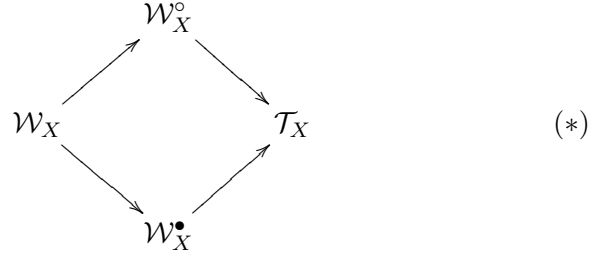
\xymatrix{
&\Law_X\ar[rd]\\
\Weak_X\ar[ru]\ar[rd]&&\tau_X&&&(*)\\
&\Zar_X\ar[ru]
}$$
where $\tau_X$ stands for the original topology of $X$ and an arrow $\tau\to\sigma$ indicates that $\tau\subset\sigma$.

A topologized semigroup $(X,\Tau_X)$ is called
\begin{itemize}
\item {\em weak$^\circ$} if $\Law_X=\Tau_X$;
\item {\em weak$^\bullet$} if $\Zar_X=\Tau_X$;
\item {\em $\II$-weak} if $\Weak_X=\Tau_X$.
\end{itemize}
In the context of topological semilattices, weak$^\circ$ topological semilattices were introduced by Lawson~\cite{Law69} (who called them semilattices with small subsemilattices) and are well-studied in Topological Algebra \cite[Ch.2]{CHK}, \cite[VI]{Bible}; $\II$-weak topological semigroups also appear naturally in the theory of topological semilattices, see \cite[Ch.2]{CHK} or \cite[VI-3.7]{Bible}. On the other hand, the notions of the weak$^\bullet$ topology $\Zar_X$ or a weak$^\bullet$ topologized semigroup seem to be new. In this paper we shall study the interplay between (separation properties) of the weak topologies $\Zar_X$, $\Law_X$, $\Weak_X$ on a topologized semilattice $X$.

The inclusion relations~$(*)$ between the topologies $\Weak_X$, $\Law_X$, $\Zar_X$ and $\Tau_X$ imply that each $\II$-weak topologized semigroup $X$ is both weak$^\circ$ and weak$^\bullet$.

To describe separation properties of the topologies  $\Law_X$, $\Zar_X$ and $\Weak_X$ let us define a topologized semigroup $X$ to be
\begin{itemize}
\item {\em $\Law$-Hausdorff} if its weak$^\circ$ topology $\Law_X$ is Hausdorff;
\item {\em $\Zar$-Hausdorff} if its weak$^\bullet$ topology $\Zar_X$ is Hausdorff;
\item {\em $\Weak$-Hausdorff} if its $\II$-weak topology $\Weak_X$ is Hausdorff;
\item {\em $\Law\Tau$-separated} if any distinct points $x,y\in X$ have disjoint neighborhoods $O_x\in\Law_X$ and $O_y\in\Tau_X$;
\item {\em $\Zar\Tau$-separated} if any distinct points $x,y\in X$ have disjoint neighborhoods $O_x\in\Zar_X$ and $O_y\in\Tau_X$;
\item {\em $\II$-separated} if for any distinct points $x,y\in X$ there exists a continuous semigroup homomorphism $h:X\to\II$ with $h(x)\ne h(y)$.
\end{itemize}
The inclusion relations $(*)$ between the topologies $\Weak_X$, $\Law_X$, $\Zar_X$, and $\Tau_X$ ensure that for a Hausdorff topologized semigroup $X$ the above separation properties relate as follows:
$$
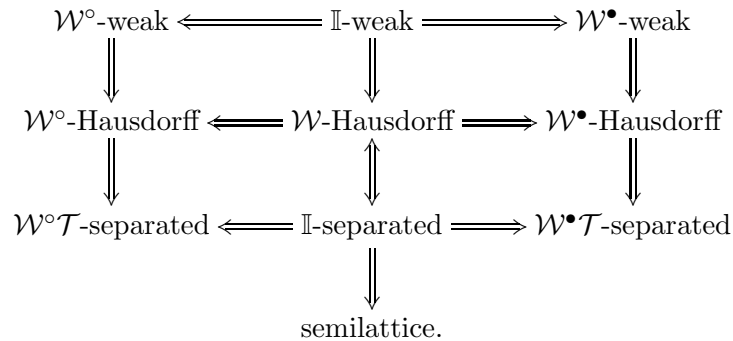
\xymatrix{
\mbox{$\Law$-weak}\ar@{=>}[d]&\mbox{$\II$-weak}\ar@{=>}[l]\ar@{=>}[r]\ar@{=>}[d]&\mbox{$\Zar$-weak}\ar@{=>}[d]\\
\mbox{$\Law$-Hausdorff}\ar@{=>}[d]&\mbox{$\Weak$-Hausdorff}\ar@{=>}[l]\ar@{=>}[r]
&\mbox{$\Zar$-Hausdorff}\ar@{=>}[d]\\
\mbox{$\Law\Tau$-separated}&\mbox{$\II$-separated}\ar@{=>}[d]\ar@{<=>}[u]\ar@{=>}[r]\ar@{=>}[l]&\mbox{$\Zar\Tau$-separated}\\
&\mbox{semilattice.}
}
$$

We complete this diagram by three properties of topologized semilattices, determined by the order and topological structure. We recall that each semilattice $X$ carries the partial order $\le$ defined by $x\le y$ iff $xy=x$.

A topologized semilattice $X$ is defined to be
\begin{itemize}
\item a {\em $U$-semilattice} if for each open set $V\subset X$ and point $x\in V$ there exists a point $v\in V$ whose upper set ${\uparrow}v=\{y\in X:vy=v\}$ contains $x$ in its interior in $X$;
\item a {\em $W$-semilattice} if for each open set $V\subset X$ and point $x\in V$ there exists a finite subset $F\subset V$ whose upper set ${\uparrow}F:=\bigcup_{y\in F}{\uparrow}y$ contains $x$ in its interior in $X$;
\item a {\em $V$-semilattice} if for any points $x\not\le y$ in $X$ there exists a point $v\notin{\downarrow}y$ in $X$ whose upper set ${\uparrow}v$ contains $x$ in its interior in $X$.
\end{itemize}
It is clear that each (Hausdorff semitopological) $U$-semilattice is a $W$-semilattice (and a $V$-semilattice).
In Proposition~\ref{p:VCH} we shall observe that each semitopological $V$-semilattice is $\Zar$-Hausdorff and in Theorem~\ref{t:UW} we shall prove that a semitopological semilattice is a $W$-semilattice if and only if it is a $U$-semilattice if and only if $X$ is a $U$-semilattice in the sense of \cite[p.16]{CHK}. By (the proof of) Lemma 2.10 in \cite{CHK}, each Hausdorff (semi)topological $U$-semilattice $X$ is $\II$-separated.

Therefore, for a Hausdorff semitopological semilattice we obtain the following Diagram~\ref{eq2} describing the interplay between various separation properties of Hausdorff semitopological semilattices.
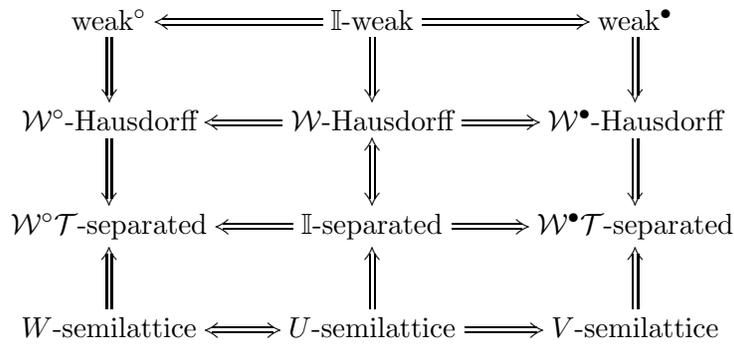
\begin{figure}[h]
\[
\xymatrix{
\mbox{weak$^\circ$}\ar@{=>}[d]&\mbox{$\II$-weak}\ar@{=>}[l]\ar@{=>}[r]\ar@{=>}[d]&\mbox{weak$^\bullet$}\ar@{=>}[d]\\
\mbox{$\Law$-Hausdorff}\ar@{=>}[d]&\mbox{$\Weak$-Hausdorff}\ar@{=>}[l]\ar@{=>}[r]\ar@{<=>}[d]
&\mbox{$\Zar$-Hausdorff}\ar@{=>}[d]\\
\mbox{$\Law\Tau$-separated}&\mbox{$\II$-separated}\ar@{=>}[l]\ar@{=>}[r]&\mbox{$\Zar\Tau$-separated}\\
\mbox{$W$-semilattice}\ar@{=>}[u]\ar@{<=>}[r]&\mbox{$U$-semilattice}\ar@{=>}[u]\ar@{=>}[r]&\mbox{$V$-semilattice}\ar@{=>}[u].
}
\]
\caption{Implications between some properties of Hausdorff semitopological semilattices}\label{eq2}
\end{figure}

One of the main results of this paper is Theorem~\ref{t:main} saying that for a compact Hausdorff semitopological semilattice all properties in Diagram~\ref{eq2} are equivalent.

Some separation properties of Diagram~\ref{eq2} remain equivalent for complete Hausdorff semitopological semilattices. A topologized semilattice $X$ is called {\em complete} if each chain in $X$ has $\inf C\in\bar C$ and $\sup C\in\bar C$ in $X$. 
Complete topologized semilattices play an important role in the theory of (absolutely) H-closed semilattices, see \cite{BBm, BBc, BBo, Bardyla-Gutik-2012, GutikRepovs2008, Stepp1969, Stepp1975}. By \cite[3.1]{BBm}, {\em a Hausdorff semitopological semilattice $X$ is complete if and only if each closed chain in $X$ is compact if and only if for any continuous homomorphism $h:S\to Y$ from a closed subsemilattice $S\subset X$ to a Hausdorff topological semigroup $Y$ the image $h(S)$ is closed in $Y$}. 

In Theorem~\ref{t:Cc} we prove that a Hausdorff semitopological semilattice $X$ is complete if and only if its weak$^\bullet$ topology $\Zar_X$ is compact.
  In Theorem~\ref{t:final} we shall prove that for a complete semitopological semilattice $X$ satisfying the separation axiom $T_1$ the following properties are equivalent:
$$\mbox{$\II$-separated $\Leftrightarrow$
$\Zar\Tau$-separated  $\Leftrightarrow$
$\Weak$-Hausdorff  $\Leftrightarrow$
$\Zar$-Hausdorff  $\Leftrightarrow$
$V$-semilattice  $\Leftrightarrow\;(\,\Weak_X=\Zar_X)$.}
$$
In the final two sections we study the interplay between the weak topologies $\Weak_X$ and $\Zar_X$ and some known intrinsic topologies on semilattices. In particular, we prove that for a (functionally) Hausdorff (semi)topological semilattice
$X$ the weak$^\bullet$ topology $\Zar_X$ coincides with the lower complete topology $\DZar_X$, which is one of the intrinsic topologies, considered by Lawson in \cite{Law}.




\section{Categorial properties of the weak$^\circ$, weak$^\bullet$, and $\II$-weak topologies}

In this section we shall establish some elementary categorial properties of the weak$^\circ$, weak$^\bullet$, and $\II$-weak topologies of topologized semigroups. 

We recall that a function $f:X\to Y$ between topological spaces is
\begin{itemize}
\item {\em continuous} if for any open set $U\subset Y$ the preimage $f^{-1}(U)$ is open in $X$;
\item {\em open} if for any open set $V\subset X$ the image $f(V)$ is open in $Y$;
\item {\em closed} if for any closed set $A\subset X$ the image $f(A)$ is closed in $Y$;
\item {\em perfect} if $f$ is closed and for every $y\in Y$ the preimage $f^{-1}(y)$ is compact;
\item a {\em topological embedding} if $f:X\to f(X)$ is a homeomorphism of $X$ onto its image $f(X)$ in $Y$.
\end{itemize}

\begin{proposition}\label{p1} Let $h:X\to Y$ be a homomorphism between topologized semigroups.
\begin{enumerate}
\item If $h$ is continuous, then $h$ remains continuous with respect to the weak$^\circ$, weak$^\bullet$, and $\II$-weak topologies on $X$ and $Y$.
\item If $h$ is open, then it remains open in the weak$^\circ$ topologies on $X$ and $Y$.
\item If $h$ is perfect, then $h$ remains perfect in the weak$^\bullet$ topologies on $X$ and $Y$.
\item If $h$ is a topological embedding and $Y$ is a subtopological semigroup, then $h$ remains a topological embedding in the weak$^\bullet$ topologies on $X$ and $Y$.
\end{enumerate}
\end{proposition}

\begin{proof} 1. Assume that the homomorphism $h$ is continuous. Then for any open (or closed) subsemigroup $S\subset Y$ the preimage $h^{-1}(S)$ is an open (or closed) subsemigroup of $X$. This implies that the preimage of any (sub)basic open set of the weak$^\circ$ (weak$^\bullet$) topology on $Y$ is open in the weak$^\circ$ (weak$^\bullet$) topology of $X$, and hence $h$ is continuous in the weak$^\circ$ (or weak$^\bullet$) topologies on $X$ and $Y$. Also for any continuous homomorphism $\varphi:Y\to\II$ the composition $\varphi\circ h:X\to\II$ is a continuous homomorphism, which implies that $h$ is continuous in the $\II$-weak topologies on $X$ and $Y$.
\smallskip

2. If the homomorphism $h$ is open, then for any set $U\in\Law_X$ in the weak$^\circ$ topology on $X$ and any point $y\in h(U)$ we can choose a point $x\in U\cap h^{-1}(y)$ and find an open subsemigroup $V\subset U$ containing $x$. Since the homomorphism $h$ is open, the image $h(V)$ is an open subsemigroup of $Y$ and hence belongs to the weak$^\circ$ topology of $Y$. Taking into account that $y=h(x)\in h(V)\subset h(U)$, we see that $y$ is an interior point of $h(U)$ in the weak$^\circ$ topology of $Y$ and hence $h(U)\in\Law_Y$. So, the homomorphism $h:(X,\Law_X)\to (Y,\Law_Y)$ is open.
\smallskip

3. Now assuming that the homomorphism $h$ is perfect, we shall show that the map $h:(X,\Zar_X)\to(Y,\Zar_Y)$ is closed. Fix any closed set $F\subset X$ in the weak$^\bullet$ topology $\Zar_X$. To show that $h(F)$ is closed in the weak$^\bullet$ topology $\Zar_Y$, take any point $y\in Y\setminus h(F)$. If $y\notin h(X)$, then $Y\setminus h(X)\in\Zar_Y$ is an open neighborhood of $y$, disjoint with $h(F)$. So, we assume that $y\in h(X)$. Since $h$ is perfect and $y\notin h(F)$, the set $h^{-1}(y)$ is compact and disjoint with $F$. Since $F$ is closed in $(X,\Zar_X)$, for every $x\in h^{-1}(y)$ there is a finite family $\F_x$ of closed subsemilattices of $X$ such that the basic open set $V_x:=X\setminus\bigcup\F_x\in\Zar_X$ is an open neighborhood of $x$, disjoint with the set $F$. For the open cover $\{V_x:x\in f^{-1}(y)\}$ of the compact subset $f^{-1}(y)$ of $X$, there exists a finite subset $\Phi\subset f^{-1}(y)$ such that $f^{-1}(y)\subset\bigcup_{x\in \Phi}V_x=X\setminus \bigcap_{x\in \Phi}\bigcup\F_x$. It follows that the intersection $\bigcap_{x\in \Phi}\bigcup\F_x$ contains $F$ and is disjoint with
$h^{-1}(y)$. Since the homomorphism $h$ is closed, for any $E\in\bigcup_{x\in \Phi}\F_x$ the image $h(E)$ is a closed subsemilattice of $Y$ and hence $h(E)$ is closed in the weak$^\bullet$ topology $\Zar_Y$ of $Y$. Then the set $C=\bigcap_{x\in\Phi}\bigcup_{E\in\F_x}h(E)$ is closed in $(Y,\Zar_Y)$, and $h(F)\subset C\subset Y\setminus\{y\}$, witnessing that the set $h(F)$ is closed in $(Y,\Zar_Y)$. Since for every $y\in Y$, the compact set $f^{-1}(y)$ remains compact in the weak$^\bullet$ topology of $X$, the closed map $h:(X,\Zar_X)\to (Y,\Zar_Y)$ is perfect.
\smallskip

4. Finally assume that $h:X\to Y$ is a topological embedding and $Y$ is a subtopological semigroup. By Proposition~\ref{p1}(1), the homomorphism $h:(X,\Zar_X)\to(Y,\Zar_Y)$ is continuous. To prove that $h$ is a topological embedding, it suffices to show that for every $\Zar_X$-closed subset $F\subset X$ and any $x\in X\setminus F$ there exists an open set $V\in\Zar_Y$ such that $h(x)\in V$ and $V\cap h(F)=\emptyset$. Since $F$ is closed in $(X,\Zar_X)$, there are closed subsemigroups $F_1,\dots,F_n$ of $X$ such that $x\in X\setminus(F_1\cup\dots \cup F_n)\subset X\setminus F$. Since $h$ is a homomorphic topological embedding, for every $i\le n$ the image $h(F_i)$ is a closed subsemigroup of $h(X)$. Since $Y$ is a subtopological semigroup, the closure $E_i$ of $h(F_i)$ in $Y$ is a subsemigroup of $Y$ such that $E_i\cap h(X)=h(F_i)$. Then $V=Y\setminus(E_1\cup\dots \cup E_n)$ is an open set in $(Y,\Zar_Y)$ such that $V\cap h(X)=h(X\setminus (F_1\cup\dots\cup F_n))$ and hence $h(x)\in V$ and $V\cap h(F)=\emptyset$.
\end{proof}

Proposition~\ref{p1}(4) implies

\begin{corollary}\label{c:subcomp} For any subtopological semigroup $Y$ and any subsemigroup $X\subset Y$ the weak$^\bullet$  topology $\Zar_X$ coincides with the subspace topology on $X$ inherited from the weak$^\bullet$  topology of $Y$.
\end{corollary}

Let us also observe the following characterization of the $T_i$-axiom of the weak$^\bullet$  topology on a subtopological band for $i\in\{0,1\}$.

We recall that a topological space $X$ satisfies the separation axiom
\begin{itemize}
\item $T_0$ if for any distinct points $x,y\in X$ there exists closed set $F\subset X$ containing exactly one of the points $x,y$;
\item $T_1$ if each singleton $\{x\}\subset X$ is closed in $X$.
\end{itemize}

\begin{proposition}\label{p:ZarT0} A subtopological band $X$ satisfies the separation axiom $T_0$ if and only if the weak$^\bullet$  topology $\Zar_X$ satisfies the separation axiom $T_0$.
\end{proposition}

\begin{proof} Since $\Zar_X\subset\Tau_X$, the $T_0$-separation property of the weak$^\bullet$  topology $\Zar_X$ implies that property of the topology $\Tau_X$.

If the topology $\Tau_X$ satisfies the separation axiom $T_0$, then for any distinct points $x,y\in X$ there exists a closed subset $F\subset X$ containing exactly one of the points $x,y$. If $x\in F$, then $\overline{\{x\}}\subset F$ is a closed subsemigroup of $X$ (because $X$ is a subtopological band) and then $E:=\overline{\{x\}}$ is a $\Zar_X$-closed set containing $x$ but not $y$. If $x\notin F$, then $E:=\overline{\{y\}}$ is a $\Zar_X$-closed set containing $x$ but not $y$.
In both cases we have found a $\Zar_X$-closed set $E$ containing exactly one of the points $x,y$. This means that the weak$^\bullet$  topology $\Zar_X$ satisfies the separation axiom $T_0$.
\end{proof}

By analogy we can prove

\begin{proposition}\label{p:ZarT1} A topologized band $X$ satisfies the separation axiom $T_1$ if and only if the weak$^\bullet$  topology $\Zar_X$ satisfies the separation axiom $T_1$.
\end{proposition}

The definitions of the weak$^\circ$ and weak$^\bullet$ topologies imply the following simple characterizations of $\Law$-Hausdorff and $\Zar$-Hausdorff topologized semigroups.

\begin{proposition} The weak$^\circ$ topology $\Law_X$ of a topologized semigroup $X$ is Hausdorff if and only if any two distinct points $x,y$ are contained in disjoint open subsemigroups of $X$.
\end{proposition}

\begin{proposition} The weak$^\bullet$ topology $\Zar_X$ of a topologized semigroup $X$ is Hausdorff if and only if for any distinct points $x,y\in X$ there exists a finite cover $\F$ of $X$ by closed subsemigroups of $X$ such no set $F\in\F$ contains both points $x$ and $y$.
\end{proposition}

\section{Shift-continuity of the $\II$-weak, weak$^\circ$ and weak$^\bullet$  topologies}

A topology $\tau$ on a semigroup $X$ is called {\em shift-continuous} if for any $a\in X$ the left shift $\ell_a:X\to X$, $\ell_a:x\mapsto ax$, and the right shift $r_a:X\to X$, $r_a:x\mapsto xa$, both are continuous. This happens if and only if $(X,\tau)$ is a semitopological semigroup.

\begin{proposition} For any semitopological semigroup $X$, the $\II$-weak topology $\Weak_X$ is shift-continuous and hence $(X,\Weak_X)$ is a semitopological semigroup.
\end{proposition}

\begin{proof} We need to prove that for every $a\in X$ the left and right shifts $\ell_a,r_a:X\to X$ are continuous with respect to the $\II$-weak topology $\Weak_X$. This will follow as soon as we prove that for any continuous homomorphism $h:X\to\II$ the compositions $h\ell_a:=h\circ \ell_a$ and $hr_a:=h\circ r_a$ are homomorphisms of $X$ into $\II$. Indeed, for any $x,y\in X$ we get
\begin{multline*}
h\ell_a(xy)=h(axy)=h(a)h(x)h(y)=h(a)h(a)h(x)h(y)=\\
=h(a)h(x)h(a)h(y)=h(ax)h(ay)=h\ell_a(x)\cdot h\ell_a(y),
\end{multline*}
which means that $h\ell_a:X\to\II$ is a continuous homomorphism. By analogy we can check that $hr_a$ is a homomorphism.
\end{proof}

\begin{example} For the discrete two-element group $X=\{1,a\}$ with generator $a$ the weak$^\circ$ and weak$^\bullet$ topologies $$\Law_X=\big\{\emptyset,\{1\},X\big\}\mbox{ \ and }\Zar_X=\big\{\emptyset,\{a\},X\big\}$$are not shift-continuous and hence $(X,\Law_X)$ and $(X,\Zar_X)$ are not semitopological semigroups. On the other hand, the $\II$-weak topology $\Weak_X=\{\emptyset,X\}$ is anti-discrete and hence is shift-continuous.
\end{example}

A topological semigroup $X$ is called {\em shift-homomorphic} if for any $a\in X$ the left shift $\ell_a:X\to X$, $\ell_a:x\mapsto ax$, and the right shift $r_a:X\to X$, $r_a:x\mapsto xa$, both are homomorphisms of $X$. The following characterization can be derived from the definitions.

\begin{proposition}\label{p2} A semigroup $X$ is shift-homomorphic if and only if $axay=axy$ and $xaya=xya$ for any $x,y,a\in X$.
\end{proposition}

Proposition~\ref{p2} implies that each semilattice is shift-homomorphic. It is easy to construct examples of shift-homomorphic semigroups which are not semilattices.

Proposition~\ref{p1}(1) implies the following useful fact.

\begin{proposition}\label{p3} If $X$ is a shift-homomorphic semitopological semigroup, then $(X,\Law_X)$ and $(X,\Zar_X)$  are semitopological semigroups.
\end{proposition}

Since semilattices are shift-homomorphic semigroups, Proposition~\ref{p3} implies

\begin{corollary}\label{c:Cst} If $X$ is a semitopological semilattice, then $(X,\Law_X)$ and $(X,\Zar_X)$ are semitopological semilattices.
\end{corollary}

\begin{problem} Is there a semitopological band $X$ for which the topologized bands $(X,\Law_X)$ and $(X,\Zar_X)$ are not semitopological?
\end{problem}

\section{Examples of weak$^\circ$ and weak$^\bullet$ topological semigroups}

In this section we present some examples of weak$^\circ$, weak$^\bullet$ and $\II$-weak topologized semigroups. We recall that a topologized semigroup $(X,\Tau_X)$ is called {\em $\II$-weak}, (resp. {\em weak$^\circ$}, {\em weak$^\bullet$}) if $\Tau_X=\Weak_X$ (resp. $\Tau_X=\Law_X$, $\Tau_X=\Zar_X)$.

A semigroup $X$ is called {\em linear} if $xy\in\{x,y\}$ for any $x,y\in X$. It is clear that each subset of a linear semigroup is a subsemigroup. Consequently,
we have

\begin{proposition} Each linear topologized semigroup is both weak$^\circ$ and weak$^\bullet$.
\end{proposition}

\begin{example} There exists a linear topological semilattice which is not $\II$-weak.
\end{example}

\begin{proof} Consider the set $X=\{0\}\cup\{\frac1n:n\in\IN\}$ endowed with the semilattice operation of minimum. Endow $X$ with the topology $\Tau_X$ consisting of sets $U\subset X$ having the property:
\begin{itemize}
\item[$(\star)$]  if $0\in U$, then $\frac1{2n}\in U$ for all but finitely many numbers $n\in\IN$.
\end{itemize}
It is clear that $X$ is a linear topological semilattice. So, $X$ is both weak$^\circ$ and weak$^\bullet$. We claim that $X$ is not $\II$-weak. In the opposite case we could find continuous homomorphisms $h_1,\dots,h_k:X\to\II$ and open sets $U_1,\dots,U_k\subset\II$ such that $$0\in \bigcap_{i=1}^k h_i^{-1}(U_i)\subset \{0\}\cup\{\tfrac1{2n}:n\in\IN\}\in\Tau_X.$$
Replacing each set $U_i$ by a smaller open set, we can assume that $U_i$ is order-convex in $\II$. Since $0$ is non-isolated in $X$, there is a number $n\in\IN$ such that $\frac1{2n}\in \bigcap_{i=1}^nh_i^{-1}(U_i)$. For every $i\le n$ the inequality $0\le\frac1{2n+1}\le\frac1{2n}$ implies $h_i(0)\le h_i(\frac1{2n+1})\le h_i(\frac1{2n})$. Since $h_i(0),h_i(\frac1{2n})\in U_i$, the order-convexity of $U_i$ ensures that $h_i(\frac1{2n+1})\in U_i$. Then $\frac1{2n+1}\in\bigcap_{i=1}^kh_i^{-1}(U_i)\subset \{0\}\cup\{\frac1{2m}:m\in\IN\}$, which is a desired contradiction showing that $\Tau_X\ne\Weak_X$.

The above example $X$ was first considered by Gutik and Repov\v s in \cite{GutikRepovs2008}.
\end{proof}

\begin{proposition}\label{p:sub} Each subsemigroup $X$ of an $\II$-weak (resp. weak$^\circ$, weak$^\bullet$) semigroup $Y$ is $\II$-weak (resp. weak$^\circ$, weak$^\bullet$).
\end{proposition}

\begin{proof} By Proposition~\ref{p1}(1), the identity embedding $i:X\to Y$ is continuous in the $\II$-weak topologies on $X$ and $Y$. If the topologized semigroup $Y$ is $\II$-weak, then $\Tau_Y=\Weak_Y$ and we have the chain of continuous identity embeddings:
$$(X,\Tau_X)\to (X,\Weak_X)\to(Y,\Weak_Y)=(Y,\Tau_Y).$$
Taking into account that the identity map $(X,\Tau_X)\to (Y,\Tau_Y)$ is a topological embedding, we conclude that the identity map $(X,\Weak_X)\to(X,\Tau_X)$ is continuous and hence $\Tau_X=\Weak_X$, which means that the topologized semigroup $X$ is $\II$-weak.

By analogy we can prove that $X$ is weak$^\circ$ (or weak$^\bullet$) if so is the topologized semigroup $Y$.
\end{proof}

\begin{proposition}\label{p:product} The Tychonoff product $X=\prod_{\alpha\in A}X_\alpha$ of\/ $\II$-weak (resp. weak$^\circ$, weak$^\bullet$) topologized semigroups is $\II$-weak (weak$^\circ$, weak$^\bullet$).
\end{proposition}

\begin{proof} Assume that the topologized semigroups $X_\alpha$, $\alpha\in A$, are $\II$-weak. By Proposition~\ref{p1}(1), for every $\alpha\in A$ the continuity of the projection $\pr_\alpha:X\to X_\alpha$ implies the continuity of $\pr_\alpha$ in the $\II$-weak topologies on $X$ and $X_\alpha$. This implies that the identity map
$$(X,\Weak_X)\to\prod_{\alpha\in A}(X_\alpha,\Weak_{X_\alpha})$$ is continuous.
Now the chain of continuous identity maps
$$(X,\Weak_X)\to\prod_{\alpha\in A}(X_\alpha,\Weak_{X_\alpha})=\prod_{\alpha\in A}(X_\alpha,\Tau_{X_\alpha})=(X,\Tau_X)\to (X,\Weak_X)$$ensures that $\Tau_X=\Weak_X$, which means that the topologized semigroup $X$ is $\II$-weak.

By analogy we can prove that the topologized semigroup $X$ is weak$^\circ$ (or weak$^\bullet$) if so are the topologized semigroups $X_\alpha$, $\alpha\in A$.
\end{proof}

Propositions~\ref{p:sub} and \ref{p:product} imply the following theorem.

\begin{theorem} Any subsemigroup of the Tychonoff product of linear topologized semigroups is both weak$^\circ$ and weak$^\bullet$.
\end{theorem}

The following lemma gives a condition under which a topologized semigroup is not weak$^\bullet$.

\begin{lemma}\label{lemma 3} A topologized semigroup $X$ is not weak$^\bullet$ if there exists an open set $U\subset X$ and a point $x\in U$ such that for every open neighborhood $O_x\subset U$ of $x$ there exists an infinite set $I\subset X\setminus O_x$ such that $ab\in O_x$ for any distinct points $a,b\in I$.
\end{lemma}

\begin{proof} Assuming that $X$ is weak$^\bullet$, we can find closed subsemigroups $F_1,\dots,F_n\subset X$ such that $x\in X\setminus(F_1\cup\dots\cup F_n)\subset U$.
By our assumption, for the open neighborhood $O_x:=X\setminus(F_1\cup\dots\cup F_n)$ of $x$ there exists an infinite set $I\subset X\setminus O_x$ such that $ab\in O_x$ for any distinct points $a,b\in I$. Since $I\subset X\setminus O_x=F_1\cup\dots\cup F_n$, for some $i\le n$ the set $I\cap F_n$ is infinite. Choose any distinct points $a,b\in I\cap F_i$ and conclude that $ab\in F_i\subset X\setminus O_x$, which contradicts the choice of $I$.
\end{proof}

Now we construct a weak$^\circ$ topological semilattice which is not weak$^\bullet$.

\begin{example}\label{ex1} Let $X$ be an infinite discrete space and $0\in X$ be any point of $X$. Endow $X$ with the semilattice operation
$$xy=\begin{cases}x&\mbox{if $x=y$},\\
0&\mbox{otherwise}.
\end{cases}
$$
It is clear that the discrete topological semilattice $X$ satisfies the condition of Lemma~\ref{lemma 3} and hence is not weak$^\bullet$. On the other hand, $X$ is weak$^\circ$, being a discrete topological semilattice. 
\end{example}

An example of a weak$^\bullet$ non-weak$^\circ$ topological semilattice is more complicated.

\begin{example}\label{ex2} There exists a countable topological semilattice $X$ which is metrizable and weak$^\bullet$ but not weak$^\circ$.
\end{example}

\begin{proof} Let $X$ be the set of all sequences $(x_n)_{n\in\w}$ of non-negative rational numbers such that $x_n=0$ for all but finitely many numbers. Endow $X$ with the semilattice operation of coordinatewise maximum. Consider the function $\Sigma:X\to\IR$, $\Sigma:(x_n)_{n\in\w}\mapsto \sum_{n\in\w}x_n$. Let $\tau$ be the smallest topology on $X$ such that for every $k\in\w$ the coordinate projection $\pr_k:X\to\IR$, $\pr_k:(x_n)_{n\in\w}\mapsto x_k$, is continuous.

Let $\Tau_X$ be the topology on $X$ generated by the base $\{U\cap\Sigma^{-1}([0,a)):U\in\tau,\;a\in\IR\}$. It can be shown that the topology $\Tau_X$ is regular and has a countable base. By the Urysohn-Tychonoff Metrizability Theorem~\cite[4.2.9]{Engelking1989}, the topological space $(X,\Tau_X)$ is metrizable.

One can check that $(X,\Tau_X)$ is a topological semilattice with respect to the semilattice operation of maximum. Moreover, this topological semilattice is weak$^\bullet$ but not weak$^\circ$.
\end{proof}

\begin{problem}\label{prob:prod} Let $X,Y$ be Hausdorff semitopological semilattices. Are the identity maps  $$
\begin{aligned}
&(X\times Y,\Law_{X\times Y})\to(X,\Law_X)\times (Y,\Law_Y),\\
&(X\times Y,\Zar_{X\times Y})\to(X,\Zar_X)\times (Y,\Zar_Y),\\
&(X\times Y,\Weak_{X\times Y})\to(X,\Weak_X)\times (Y,\Weak_Y)
\end{aligned}$$homeomorphisms?
\end{problem}

In Proposition~\ref{p:prod2} we shall give a partial affirmative answer to this problem for $\II$-separated complete semitopological semilattices.

\section{$\Zar$-compact topologized semigroups}

In this section we study topologized semigroups whose weak$^\bullet$ topology is compact. Such topologized semigroups are called {\em $\Zar$-compact}.

First, we present a simple characterization of $\Zar$-compact topologized semigroups and derive from it a more complicated (and useful) characterization of $\Zar$-compact semitopological semilattices.

A non-empty family $\F$ of subsets of a set $X$ is called {\em centered} if each non-empty finite subfamily $\mathcal E\subset\F$ has non-empty intersection $\bigcap\mathcal E$. The classical Alexander's subbase Theorem \cite[3.12.2]{Engelking1989} says that a topological space $X$ is compact if and only if $X$ has a subbase $\mathcal B$ of the topology such that any centered subfamily $\F\subset\{X\setminus B:B\in\mathcal B\}$ has non-empty intersection. This Alexander's Theorem implies the following characterization.

\begin{theorem}\label{t:Cc1} A topologized semigroup $X$ is $\Zar$-compact if and only if each centered family of closed subsemigroups of $X$ has non-empty intersection.
\end{theorem}

Now we shall detect topologized semilattices whose weak$^\bullet$ topology is compact. First we recall some definitions related to topologized posets.

By a {\em topologized poset} we understand a poset endowed with a topology.
A topologized poset $X$ is called
\begin{itemize}
\item {\em complete} if each non-empty chain $C\subset X$ has $\inf C\in\bar C$ and $\sup C\in\bar C$;
\item {\em chain-compact} if each closed chain in $X$ is compact;
\item {\em ${\uparrow}{\downarrow}$-closed} if for any point $x\in X$ the sets ${\uparrow}x$ and ${\downarrow}x$ are closed in $X$.
\end{itemize}

\begin{lemma}\label{l:cl-chain} Let $X$ be a ${\uparrow}{\downarrow}$-closed topologized poset. The closure $\bar C$ of any chain $C$ in $X$ is a chain.
\end{lemma}

\begin{proof} Assuming that $\bar C$ is not a chain, we can find two points $x,y\in \bar C$ such that $x\notin{\updownarrow}y:=({\uparrow}y)\cup({\downarrow}y)$. Since $X$ is ${\uparrow}{\downarrow}$-closed, the set $X\setminus{\updownarrow}y$ is an open neighborhood of $x$. Since $x\in \bar C$, there exists a point $z\in C\setminus{\updownarrow}y$. Taking into account that $C$ is a chain, we conclude that $C\subset{\updownarrow}z$ and hence $y\in \bar C\subset \overline{{\updownarrow}z}={\updownarrow}z$, which contradicts $z\notin{\updownarrow}y$.
\end{proof}

Complete topologized posets are studied in detail in \cite{BBo}, where it is proved that completeness of topologized posets can be equivalently defined using up-directed and down-directed sets instead of chains.

A subset $D$ of a poset $X$ is called {\em up-directed} (resp. {\em down-directed}) if for any $x,y\in D$ there exists $z\in D$ such that $x\le z$ and $y\le z$ (resp. $z\le x$ and $z\le y$).

The following characterization was proved in \cite{BBo}. It is a topological version of a known characterization of complete posets due to Iwamura \cite{Iwamura} (see also \cite{Bruns} and \cite{Mark}).

\begin{lemma}\label{l:Iwa} A topologized poset $X$ is complete if and only if each non-empty up-directed set $U\subset X$ has $\sup U\in\bar U$ and each non-empty down-directed set $D\subset X$ has $\inf D\in\bar D$ in $X$.
\end{lemma}

A topologized semilattice $X$ is called {\em complete} if it is complete as a topologized poset endowed with the natural partial order $\le$ (defined by $x\le y$ iff $xy=x$).

\begin{lemma}\label{l:c=>Zar} Each complete topologized semilattice is $\Zar$-compact.
\end{lemma}

\begin{proof} By Theorem~\ref{t:Cc1}, to show that $X$ is $\Zar$-compact, it suffices to prove that each centered family $\F$ of closed subsemilattices in $X$ has non-empty intersection.

This will be proved by induction on the cardinality $|\F|$ of the family $\F$. If $\F$ is finite, then $\bigcap\F\ne\emptyset$ as $\F$ is centered. Assume that for some infinite cardinal $\kappa$ we have proved that each centered family $\F$ consisting of $|\F|<\kappa$ many closed subsemilattices of $X$ has non-empty intersection.

Take any centered family $\F=\{F_\alpha\}_{\alpha\in\kappa}$ of closed subsemilattices of $X$. By the inductive assumption, for every $\alpha<\kappa$ the closed subsemilattice $F_{\le\alpha}=\bigcap_{\beta\le\alpha}F_\beta$ is not empty. Let $M_\alpha$ be a maximal chain in $F_{\le\alpha}$. By the completeness of $X$, the chain $M_\alpha$ has $\inf M_\alpha\in\bar M_\alpha\subset \bar F_{\le\alpha}=F_{\le\alpha}$. We claim that the element $x_\alpha:=\inf M_\alpha$ is the smallest element of the semilattice $F_{\le\alpha}$. In the opposite case, we can find an element $z\in F_{\le\alpha}$ such that $x_\alpha\not\le z$ and conclude that $zx_\alpha<x_\alpha$ and $\{zx_\alpha\}\cup M_\alpha$ is a chain in $F_{\le \alpha}$ that properly contains the maximal chain $M_\alpha$, which contradicts the maximality of $M_\alpha$. This contradiction shows that $x_\alpha$ is the smallest element of the semilattice $F_{\le\alpha}$.

Observe that for any ordinals $\alpha\le\beta<\kappa$ the inclusion $x_\beta\in F_{\le\beta}\subset F_{\le\alpha}$ implies $x_\alpha\le x_\beta$. So, for every $\alpha\in\kappa$ the set $C_\alpha:=\{x_\beta\}_{\alpha\le\beta<\kappa}$ is a chain in $X$. By the completeness of $X$, the chain $C_\alpha$ has $\sup C_\alpha\in\bar C_\alpha\subset F_{\le\alpha}$. Since the transfinite sequence $(x_\alpha)_{\alpha\in\kappa}$ is non-decreasing, $\sup C_\alpha=\sup C_0$ for all $\alpha\in\kappa$. Then $\sup C_0\in\bigcap_{\alpha\in\kappa}F_{\le\alpha}=\bigcap_{\alpha\in\kappa}F_\alpha$, which means that the family $\F=\{F_\alpha\}_{\alpha\in\kappa}$ has non-empty intersection.
\end{proof}

\begin{lemma}\label{l:Zar=>cc} Each $\Zar$-compact topologized semilattice $X$ is chain-compact.
\end{lemma}

\begin{proof} We should prove that any closed chain $C$ in $X$ is compact.
This will follow as soon as we show that each centered family $\F$ of closed subsets of $C$ has non-empty intersection. Since $C$ is a chain, each set $F\in\F$ is a closed subsemilattice of $X$. By Theorem~\ref{t:Cc1}, the $\Zar$-compactness of the topologized semilattice $X$ implies that $\bigcap\F\ne\emptyset$.
\end{proof}

 Lemmas~\ref{l:c=>Zar} and \ref{l:Zar=>cc} will be used to prove the following characterization of $\Zar$-compactness.

\begin{theorem}\label{t:Cc} For an ${\uparrow}{\downarrow}$-closed topologized semilattice $X$ the following conditions are equivalent:
\begin{enumerate}
\item $X$ is complete;
\item $X$ is $\Zar$-compact;
\item $X$ is chain-compact.
\end{enumerate}
\end{theorem}

\begin{proof} The implications $(1)\Ra(2)\Ra(3)$ follow from Lemmas  \ref{l:c=>Zar} and \ref{l:Zar=>cc}. To prove that $(3)\Ra(1)$, assume that an ${\uparrow}{\downarrow}$-closed topologized semilattice $X$ is chain-compact and take any chain $C\subset X$. By Lemma~\ref{l:cl-chain}, the closure $\bar C$ of $C$ in $X$ is a chain. By the chain-compactness of $X$, the closed chain $\bar C$ is compact. By the compactness of $\bar C$ and the ${\uparrow}{\downarrow}$-closedness of $X$, the centered family $\{\bar C\cap{\downarrow}x:x\in \bar C\}$ of closed sets in $\bar C$ has non-empty intersection, consisting of a unique point $\min\bar C$, which is the smallest element of the chain $\bar C$. It is clear that $\min\bar C$ is a lower bound of the set $C$. For any other lower bound $b\in X$ of $C$ we get $C\subset{\uparrow}b$ and hence $\min\bar C\in\bar C\subset\overline{{\uparrow}b}={\uparrow}b$, which means that $b\le\min\bar C$ and $\inf C=\min \bar C\in\bar C$. By analogy we can prove that the chain has $\sup C=\max\bar C\in\bar C$.
\end{proof}

\section{On $U$-semilattices, $W$-semilattices, and $V$-semilattices}

In this section we shall prove some characterizations involving $U$-semilattices, $W$-semilattices, and $V$-semilattices, which were defined in the introduction.

It is clear that each $U$-semilattice is a $W$-semilattice. The converse is true for semitopological semilattices.

\begin{theorem}\label{t:UW} For a semitopological semilattice $X$ the following conditions are equivalent:
\begin{enumerate}
\item $X$ is a $U$-semilattice;
\item $X$ is a $W$-semilattice;
\item for any open set $U={\uparrow}U$ of $X$ and any point $x\in U$ there exists a point $y\in U$ whose upper set ${\uparrow}y$ is a neighborhood of $x$ in $X$.
\end{enumerate}
\end{theorem}

\begin{proof} The implications $(2)\Leftarrow(1)\Ra(3)$ are trivial.
\smallskip

$(3)\Ra(1)$ Assume that the semilattice $X$ satisfies the condition (3). To show that $X$ is a $U$-semilattice, fix any open set $V\subset X$ and a point $x\in V$. We need to find a point $v\in V$ whose upper set ${\uparrow}v$ contains $x$ in its interior.

The separate continuity of the semilattice operation implies that the set ${\uparrow}V=\bigcup_{v\in V}\ell^{-1}_v(V)$ is open. By the condition (3), there exists a point $y\in{\uparrow}V$ whose upper set ${\uparrow}y$ contains the point $x$ in its interior. For the point $y\in{\uparrow}V$ find a point $v\in V$ with $v\le y$ and observe that the upper set ${\uparrow}v\supset{\uparrow}y$ contains $x$ in its interior.
\smallskip

$(2)\Ra(1)$ Assume that $X$ is a $W$-semilattice. To show that $X$ is a $U$-semilattice, take any open set $V\subset X$ and any point $x\in V$.
We need to find an element $e\in V$ whose upper set ${\uparrow}e$ is a neighborhood of $x$ in $X$.

Since $X$ is a $W$-semilattice, there exists a finite set $F\subset V$ whose upper set ${\uparrow}F$ is a neighborhood of $x$. By Lemma~\ref{l:W}, for some $e\in F\subset V$ the upper set ${\uparrow}e$ is a neighborhood of $x$.
\end{proof}

\begin{lemma}\label{l:W} Let $X$ be a semitopological semilattice and $F\subset X$ be a finite subset whose upper set ${\uparrow}F$ contains a point $x\in X$ in its interior. Then for some $e\in F$ the upper set ${\uparrow}e$ is a neighborhood of $x$.
\end{lemma}

\begin{proof} Since ${\uparrow}F\cap{\downarrow}x$ is a neighborhood of $x$ in ${\downarrow}x$, it is not nowhere dense in ${\downarrow}x$. Since the finite union of nowhere dense sets is nowhere dense, for some element $e\in F$ the closed set ${\uparrow}e\cap{\downarrow}x$ has non-empty interior $W$ in ${\downarrow}x$. We claim that $x\in W$. Indeed, for any $w\in W$ we get $w=wx$ and by the continuity of the shift $\ell_w:{\downarrow}x\to{\downarrow}x$, the set $O_x:=\{z\in{\downarrow}x:zw\in W\}$ is an open neighborhood of $x$ in ${\downarrow}x$. It is clear that $O_x\subset{\uparrow}W\cap{\downarrow}x\subset {\uparrow}e\cap{\downarrow}x$ and hence $O_x\subset W$.

By the continuity of the shift $\ell_x:X\to{\downarrow}x$, $\ell_x:y\mapsto xy$, the set $\ell_x^{-1}(W)=\{y\in X:xy\in W\}$ is an open neighborhood of $x$ in $X$ such that $\ell_x^{-1}(W)\subset{\uparrow}W\subset {\uparrow}e$. This means that for the point $e\in F\subset U$ the upper set ${\uparrow}e$ is a neighborhood of $x$.
\end{proof}

The equivalence $(1)\Leftrightarrow(3)$ of Theorem~\ref{t:UW} means that our definition of a $U$-semilattice is equivalent to the classical definition of a $U$-semilattice given in \cite[p.16]{CHK}. The (proof of) Lemma 2.10 in \cite{CHK} yields the following important fact.

\begin{theorem}[Lawson]\label{t:UI} Each Hausdorff semitopological $U$-semilattice is $\II$-separated.
\end{theorem}

Next, we establish some properties of $V$-semilattices.

\begin{proposition}\label{p:VCH} Each semitopological $V$-semilattice $X$ satisfying the separation axiom $T_1$ is $\Zar$-Hausdorff.
\end{proposition}

\begin{proof} To prove that the $V$-semilattice $X$ is $\Zar$-Hausdorff, fix any distinct points $x,y\in X$. Because of symmetry, we can assume that $x\not\le y$. Since $X$ is a $V$-semilattice, there exists a point $v\not\in{\downarrow}y$ whose upper set ${\uparrow}v$ contains $x$ in its interior $V$ in $X$. By the continuity of the left shifts on $X$, the set ${\uparrow}V=\bigcup_{u\in V}\ell_u^{-1}(V)$ is open in $X$. Taking into account that ${\uparrow}V\subset{\uparrow}v$, we conclude that ${\uparrow}V=V$ and then $F=X\setminus V$ is a closed subsemilattice of $X$.

Since $X$ is a $T_1$-space, the singleton $\{v\}$ is closed and so is its preimage ${\uparrow}v=\ell^{-1}_v(\{v\})$ under the continuous shift $\ell_v$.
Then $O_x:=X\setminus F$ and $O_y=X\setminus{\uparrow}v$ are disjoint $\Zar_X$-open neighborhoods of $x$ and $y$, respectively.
\end{proof}

We shall say that topologized semilattice $X$ is {\em ${\downarrow}$-chain-compact} if for each $x\in X$ the subsemilattice ${\downarrow}x:=\{y\in X:y\le x\}$ of $X$ is chain-compact.

\begin{theorem}\label{t:Vc} For a $\downarrow$-chain-compact semitopological semilattice the following conditions are equivalent:
\begin{enumerate}
\item $X$ is $\Zar$-Hausdorff;
\item $X$ is $\Zar\Tau$-separated;
\item $X$ is a $V$-semilattice satisfying the separation axiom $T_1$.
\end{enumerate}
\end{theorem}

\begin{proof} The implication $(1)\Ra(2)$ is trivial and $(3)\Ra(1)$ was proved in Proposition~\ref{p:VCH}. To prove that $(2)\Ra(3)$, assume that $X$ is $\Zar\Tau$-separated. Then $X$ is Hausdorff and hence satisfies the separation axiom $T_1$. Moreover, the separate continuity of the semilattice opeartion implies that $X$ is  ${\uparrow}{\downarrow}$-closed. To prove that $X$ is a $V$-semilattice, take any two points $x\not\le y$ in $X$. We need to find a point $e\notin{\downarrow}y$ in $X$ whose upper set ${\uparrow}e$ is a neighborhood of $x$.

By our assumption, the subsemilattice ${\downarrow}y$ is chain-compact and by Theorem~\ref{t:Cc}, ${\downarrow}y$ is $\Zar$-compact. By Corollary~\ref{c:subcomp},  the set ${\downarrow}y$ is compact in the topological space $(X,\Zar_X)$.

Since the semilattice $X$ is $\Zar\Tau$-separated, for every point $z\in{\downarrow}y$ there exist disjoint open sets $O_z\in\Zar_X$ and $U_{x,z}\in\Tau_X$ such that $z\in O_z$ and $x\in U_{x,z}$. We can assume that $O_z$ is of basic form $O_z=X\setminus\bigcup\F_z$ for a finite family $\F_z$ of closed subsemilattices of $X$.

By the compactness of ${\downarrow}y$, the $\Zar_X$-open cover $\{O_z:z\in{\downarrow}y\}$ of ${\downarrow}y$ has a finite subcover $\{O_z:z\in E\}$ (here $E$ is a suitable finite subset of ${\downarrow}y$).
Then ${\downarrow}y\subset\bigcup_{z\in E}O_z=\bigcup_{z\in E}X\setminus\bigcup\F_z=X\setminus \bigcap_{z\in E}\bigcup\F_z$ and the $\Zar_X$-open set $X\setminus \bigcap_{z\in E}\bigcup\F_z$ does not intersect the neighborhood $U_x:=\bigcap_{z\in E}U_{x,z}$ of $x$.

It follows that $U_x\subset \bigcap_{z\in E}\bigcup\F_y\subset X\setminus{\downarrow}y$. Observe that $\bigcap_{z\in E}\bigcup\F_y=\bigcup\mathcal K$ where $\mathcal K=\{\bigcap_{z\in E}F_z:(F_z)_{z\in E}\in\prod_{z\in E}\F_z\}$. Each non-empty set $K\in\mathcal K$ is a closed subsemilattice in $X$ which has the smallest element $x_K$ by the ${\downarrow}$-chain-compactness of $X$. Consider the finite set $F=\{x_K:K\in\mathcal K\setminus\{\emptyset\}\}$ and observe that
$U_x\subset\bigcup\mathcal K\subset{\uparrow}F\subset X\setminus{\downarrow}y$.
By Lemma~\ref{l:W}, the set $F$ contains a point $e$ such that ${\uparrow}e$ is a neighborhood of $x$. It follows from $e\notin{\downarrow}y$ that $e\not\le y$.
\end{proof}

\section{Separation properties of weak topologies on compact topological semilattices}

In this section we prove that all separation properties in Diagram~\ref{eq2} are equivalent for compact Hausdorff semitopological semilattices. It should be mentioned that each compact Hausdorff semitopological semilattice is a topological semilattice, see \cite{Law74}. 

The equivalence of the conditions (1,2,4,6,7,9,11) in the following theorem is a well-known result of Lawson \cite{Law69}, see also \cite[VI-3.4]{Bible}.

\begin{theorem}\label{t:main} For a compact Hausdorff semitopological semilattice $X$ with topology $\Tau_X$ the following conditions are equivalent:
\begin{enumerate}
\item $\Weak_X=\Tau_X$;
\item $\Law_X=\Tau_X$;
\item $\Zar_X=\Tau_X$;
\item the weak$^\circ$ topology $\Law_X$ is Hausdorff;
\item the weak$^\bullet$ topology $\Zar_X$ is Hausdorff;
\item the $\II$-weak topology $\Weak_X$ is Hausdorff;
\item $X$ is $\Law\Tau$-separated;
\item $X$ is $\Zar\Tau$-separated;
\item $X$ is $\II$-separated;
\item $X$ is a $W$-semilattice;
\item $X$ is a $U$-semilattice;
\item $X$ is a $V$-semilattice.
\end{enumerate}
\end{theorem}

\begin{proof} It suffices to prove the chains of implications $(2)\Ra(4)\Ra(7)\Ra(2)$, $(11)\Ra(12)\Ra(3)\Leftrightarrow(5)\Leftrightarrow(8)\Ra(10)\Ra(11)$, and $(1)\Ra(2)\Ra(11)\Ra(9)\Ra(6)\Ra(1)$.
\smallskip

The implications $(2)\Ra(4)\Ra(7)$ trivially follow from the inclusions of the topologies $\Law_X\subset\Tau_X$.
\smallskip

$(7)\Ra(2)$: Assume that the compact semitopological semigroup $X$ is $\Law\Tau$-separated. Since $\Law_X\subset\Tau_X$, the equality $\Law_X=\Tau_X$ will follow as soon as we check for any point $x\in X$ and neighborhood $U\in\Tau_X$ of $X$ there exists a neighborhood $U_x\in\Law_X$ of $x$ such that $U_x\subset U$. Since $X$ is $\Law\Tau$-separated, for any $y\in X\setminus U$ we can find disjoint neighborhoods $U_{x,y}\in\Law_X$ and $O_y\in\Tau_X$ of the points $x$ and $y$, respectively. By the compactness of $X$ the open cover $\{O_y:y\in X\setminus U\}$ of the closed set $X\setminus U\subset X$ has a finite subcover $\{O_y:y\in F\}$. Then $U_x:=\bigcap_{y\in F}U_{x,y}$ is a $\Law_X$-open neighborhood of $x$, contained in $U$.
\smallskip

The implication $(11)\Ra(12)$ is trivial and $(12)\Ra(5)\Leftrightarrow(8)$ were proved in Theorem~\ref{t:Vc}. The equivalence $(3)\Leftrightarrow(5)$ follows from the inclusion $\Zar_X\subset\Tau_X$ and the compactness of the topology $\Tau_X$.
\smallskip

$(8)\Ra(10)$: Assume that $\Zar_X=\Tau_X$. To prove that $X$ is a $W$-semilattice, take any open set $U\subset X$ and point $x\in U$. We need to find a finite subset $F\subset U$ whose upper set ${\uparrow}F$ is a neighborhood of $x$.

Since the space $X$ is compact and Hausdorff, the point $x$ has a compact neighborhood $K_x\subset U$. Since $\Zar_X=\Tau_X$, each point $y\in X\setminus U$ has an open neighborhood $V_y\in\Zar_X$ which is disjoint with the compact set $K_x$. We can assume that $V_y=X\setminus \bigcup\F_y$ for a finite family $\F_y$ of closed subsemilattices of $X$. By the compactness of $X\setminus U$, the open cover $\{V_y:y\in X\setminus U\}$ has a finite subcover $\{V_y:y\in E\}$. Here $E\subset X\setminus U$ is a suitable finite set.

It follows that $X\setminus U\subset \bigcup_{y\in E}V_y=\bigcup_{y\in E}(X\setminus\bigcup\F_y)=X\setminus\bigcap_{y\in E}\bigcup\F_y\subset X\setminus K_x$ and hence $K_x\subset \bigcap_{y\in E}\bigcup\F_y\subset U$.

Observe that $\bigcap_{y\in E}\bigcup\F_y=\bigcup\mathcal K$ where $\mathcal K=\{\bigcap_{y\in E}F_y:(F_y)_{y\in E}\in\prod_{y\in E}\F_y\}$. Each set $K\in\mathcal K\setminus\{\emptyset\}$ is a non-empty compact subsemilattice of $X$, so $K\subset {\uparrow}x_K$ where $x_K\in K$ is the smallest element of $K$ (which exists by the compactness of $K$). It follows that $F=\big\{x_K:K\in\mathcal K\setminus\{\emptyset\}\big\}$ is a finite subset of $\bigcup\mathcal K\subset U$ such that $K_x\subset\bigcup\mathcal K\subset {\uparrow}F$, which means that ${\uparrow}F$ is a neighborhood of $x$.
\smallskip

The implication $(10)\Ra(11)$ was proved in Theorem~\ref{t:UW}.
\smallskip

The implication $(1)\Ra(2)$ trivially follows from the inclusions $\Weak_X\subset\Law_X\subset\Tau_X$.
\smallskip

$(2)\Ra(11)$: By \cite{Law74}, the compact Hausdorff semitopological semilattice $(X,\Tau_X)$ is a topological semilattice. If $X$ is weak$^\circ$, then by Theorem 2.12 \cite{CHK} and Theorem~\ref{t:UW}, $X$ is a $U$-semilattice.
\smallskip

$(11)\Ra(9)$: If $X$ is a $U$-semilattice, then $X$ is $\II$-separated by Theorem~\ref{t:UI}.
\smallskip

The implication $(9)\Ra(6)$ trivially follows from the definitions.
\smallskip

$(6)\Ra(1)$: If $X$ is $\Weak$-Hausdorff, then the identity map $(X,\Tau_X)\to(X,\Weak_X)$ is a homeomorphism by the compactness of $X$ and the Hausdorff property of $\Weak_X$.
\end{proof}

\begin{remark} Examples~\ref{ex1} and \ref{ex2} show that the equivalence of the conditions (2) and (3) in Theorem~\ref{t:main} does not hold for non-compact topological semilattices.
\end{remark}


Some statements of Theorem~\ref{t:main} remain equivalent for complete semitopological semilattices.


\begin{theorem}\label{t:final} For a complete semitopological semilattice $X$ the following conditions are equivalent:
\begin{enumerate}
\item the topology $\Zar_X$ is Hausdorff;
\item the topology $\Weak_X$ is Hausdorff;
\item $X$ is $\II$-separated;
\item $X$ is a $V$-semilattice satisfying the separation axiom $T_1$;
\item $X$ is a $T_0$-space and $\Weak_X=\Zar_X$.
\end{enumerate}
\end{theorem}

\begin{proof} The equivalence $(2)\Leftrightarrow(3)$ is trivial and $(2)\Ra(1)$ follows from the inclusion $\Weak_X\subset\Zar_X$.
\smallskip

$(1)\Ra(3)$: Assume that the topology $\Zar_X$ is Hausdorff. Then $X$ is Hausdorff and by Corollary~\ref{c:Cst} and Theorem~\ref{t:Cc}, $(X,\Zar_X)$ is a Hausdorff compact semitopological semilattice. Since each $\Zar_X$-closed subsemilattice in $X$ is $\Tau_X$-closed, the semitopological semilattice $(X,\Zar_X)$ is weak$^\bullet$ and by Theorem~\ref{t:main}, $\II$-separated. Since $\Zar_X\subset\Tau_X$, the semitopological semilattice $(X,\Tau_X)$ is $\II$-separated, too.
\smallskip

The implication $(1)\Ra(4)$ follows from Theorems~\ref{t:Vc} and \ref{t:Cc}. The implication $(4)\Ra(1)$ is proved in Proposition~\ref{p:VCH}.
\smallskip

$(2)\Ra(5)$: Assume that  the topology $\Weak_X$ is Hausdorff. Then $X$ is Hausdorff and hence satisfies the separation axiom $T_0$. By Theorem~\ref{t:Cc}, the weak$^\bullet$ topology $\Zar_X$ is compact, which implies that the identity map $(X,\Zar_X)\to (X,\Weak_X)$ to the Hausdorff space $(X,\Weak_X)$ is a homeomorphism and hence $\Zar_X=\Weak_X$.
\smallskip

$(5)\Ra(3)$: Assume that $X$ is a $T_0$-space and $\Zar_X=\Weak_X$. By Proposition~\ref{p:ZarT0}, the weak$^\bullet$ topology $\Zar_X$ satisfies the separation axiom $T_0$. Assuming that $X$ is not $\II$-separated, we can find two distinct points $x,y\in X$ such that $h(x)=h(y)$ for any continuous homomorphism $h:X\to\II$. Then the points $x,y$ cannot be separated by $\Weak_X$-open sets and $\Weak_X$ does not satisfy the separation axiom $T_0$. So, $\Weak_X\ne\Zar_X$, which is a desired contradiction.
\end{proof}

\begin{proposition} If a Hausdorff complete  semitopological semilattice $X$ is weak$^\bullet$, then $X$ is a weak$^\circ$ compact topological semilattice.
\end{proposition}

\begin{proof} By Theorem~\ref{t:Cc}, the weak$^\bullet$ topology $\Zar_X$ on $X$ is compact. If $X$ is weak$^\bullet$, then $X$ is compact and by Theorem~\ref{t:main} the compact semitopological semilattice $X$ is weak$^\circ$. By \cite{Law74}, the compact Hausdorff semitopological semilattice $X$ is a topological semilattice.
\end{proof}

The following proposition gives a partial affirmative answer to Problem~\ref{prob:prod}.

\begin{proposition}\label{p:prod2} If $X=\prod_{\alpha\in A}X_\alpha$ is the Tychonoff product of\/  $\Zar\Tau$-separated complete semitopological semilattices, then the identity map
$$(X,\Zar_X)=(X,\Weak_X)\to\prod_{\alpha\in A}(X_\alpha,\Weak_{X_\alpha})=\prod_{\alpha\in A}(X_\alpha,\Zar_{X_\alpha})$$ is a homeomorphism.
\end{proposition}

\begin{proof} By Theorem~\ref{t:final}, the topologized semilattices $X_\alpha$, $\alpha\in A$, are $\II$-separated and so is their Tychonoff product $X$. Also Theorem~\ref{t:final} implies that $\Zar_X=\Weak_X$ and $\Zar_{X_\alpha}=\Weak_{X_\alpha}$ for all $\alpha\in A$. By \cite{BBo}, the Hausdorff semitopological semilattice $X$ is complete and by Theorem~\ref{t:Cc}, the weak$^\bullet$ topology $\Zar_X$ is compact and so is the $\II$-weak topology $\Weak_X=\Zar_X$. By Proposition~\ref{p1}(1), the identity map
$$\mathrm{id}:(X,\Weak_X)\to\prod_{\alpha\in A}(X_\alpha,\Weak_{X_\alpha})$$
is continuous and hence is a homeomorphism (by the compactness of the topology $\Weak_X$ and the Hausdorff property of the topologies $\Weak_{X_\alpha}$, $\alpha\in A$).
\end{proof}

\section{Applications to intrinsic topologies on semilattices}

In this section we apply the results of the preceding sections to studying intrinsic topologies on posets and semilattices.

We recall that a {\em poset} is a set endowed with  a partial order.  A topology $\tau$ on a poset $X$ is called {\em intrinsic} if each order isomorphism of $X$ is a homeomorphism in the topology $\tau$. Each poset $X$ carries many natural intrinsic topologies, see \cite{Law}, \cite{Vai}. Let us recall just five of them: the chain topology $\Chain_X$, the Dedekind topology $\Dedekind_X$, the interval topology $\Interval_X$, the Scott topology $\Scott_X$, and the Lawson topology $\Lawson_X$.

Let $X$ be a poset. A subset $A\subset X$ is called
\begin{itemize}
\item {\em chain-closed} if $A$ contains $\inf C$ and $\sup C$ of any chain $C\subset A$ that has $\inf C$ and $\sup C$ in $X$;
\item {\em chain-open} if the complement $X\setminus A$ is chain-closed in $X$.
\end{itemize}
It can be shown that finite unions and arbitrary intersections of chain-closed sets in $X$ are chain-closed, which implies that the family $\Chain_X$ of chain-open sets in $X$ is a topology,  called the {\em chain topology} of $X$, see \cite{Law} (in \cite{Vai} the chain topology was called the finest $i$-topology).
\smallskip

Next, we define the Dedekind topology $\Dedekind_X$ on $X$.
A subset $A$ of a poset $X$ is called
\begin{itemize}
\item {\em up-closed} if $A$ contains $\sup D$ of any up-directed subset $D\subset A$ having $\sup D$ in $X$;
\item {\em down-closed} if $A$ contains $\inf D$ of any down-directed subset $D\subset A$ with $\inf D$ in $X$;
\item {\em Dedekind closed} if $A$ is up-closed and down-closed in $X$;
\item {\em Dedekind open} if $X\setminus A$ is chain-closed in $X$.
\end{itemize}
It can be shown that finite unions and arbitrary intersections of Dedekind closed sets in $X$ are Dedekind closed, which implies that the family $\Dedekind_X$ of Dedekind open sets in $X$ is a topology, which will be called the {\em Dedekind  topology} on $X$, see \cite{Law}.

Since for any $x\in X$ the upper and lower sets ${\uparrow}x$ and ${\downarrow}x$ are Dedekind closed in $X$, the Dedekind topology $\Dedekind_X$ contains the {\em interval topology} $\Interval_X$ on $X$, generated by the subbase
$\{X\setminus{\downarrow}x,X\setminus{\uparrow}x:x\in X\}$.

The {\em Scott topology} $\Scott_X$ on a poset $X$ consists of Dedekind open sets $U={\uparrow}U$ where ${\uparrow}U=\bigcup_{x\in U}{\uparrow}x$.

The {\em Lawson topology} $\Lawson_X$ is generated by the subbase $\Scott_X\cup\{X\setminus{\uparrow}x:x\in X\}$.

It is clear that $$\Interval_X\subset\Lawson_X\subset\Dedekind_X\subset\Chain_X$$
for any poset $X$. More information on the Scott and the Lawson topologies can be found in \cite{Bible}.

We shall show that for a complete poset $X$ the topologies $\Dedekind_X$ and $\Chain_X$ coincide.

A poset $X$ is called {\em complete} if each non-empty chain $C\subset X$ has $\inf C$ and $\sup C$ in $X$. Observe that a poset $X$ is complete if and only if $X$ endowed with the antidiscrete topology is a complete topologized poset. This observation, combined with Lemma~\ref{l:Iwa} implies the following (known) characterization.

\begin{lemma}\label{l:complete-poset} A poset $X$ is complete if and only if each non-empty up-directed set $U\subset X$ has $\sup U\in X$ and each down-directed set $D\subset X$ has $\inf D\in X$.
\end{lemma}

\begin{lemma} For a complete poset $X$ we have $\Dedekind_X=\Chain_X$.
\end{lemma}

\begin{proof} Since $\Dedekind_X\subset\Chain_X$, it suffices to show that each chain-closed set $A\subset X$ is up-closed and down-closed.

By the completeness of $X$,  every non-empty up-directed subset $U\subset A$ has $\sup U\in X$. It remains to show that $\sup U\in A$. This will be proved by induction on the cardinality $|U|$ of $U$. If $U$ is finite, then $\sup U\in U\subset A$. Assume that for some infinite cardinal $\kappa$ we have proved that each up-directed subset $U\subset A$ of cardinality $|U|<\kappa$ the point $\sup U$ (which exists by the completeness of $X$) belongs to $A$. Take any non-empty up-direceted subset $U\subset A$ of cardinality $|U|=\kappa$. Write $U$ as the union $U=\bigcup_{\alpha<\kappa}U_\alpha$ of an increasing transfinite sequence $(U_\alpha)_{\alpha<\kappa}$ consisting of up-directed sets $U_\alpha$ of cardinality $|U_\alpha|<\kappa$. By the completeness of $X$ and Lemma~\ref{l:complete-poset}, each up-directed poset $U_\alpha$ has $\sup U_\alpha\in X$ and by the inductive assumption, $\sup U_\alpha\in A$. Since the transfinite sequence $(U_\alpha)_{\alpha\in\kappa}$ is increasing, the set $C:=\{\sup U_\alpha\}_{\alpha\in\kappa}$ is a chain in $A$. Since $X$ is complete, the chain $C$ has $\sup C$ in $X$ and since $A$ is chain-closed in $X$, $\sup C\in A$. It remains to observe that $\sup C=\sup U$.

By analogy we can prove that the set $A$ is down-closed in $X$ and hence $A$ is Dedekind closed in $X$.
\end{proof}

Now assume that a poset $X$ is a semilattice (which means that any finite subset $F\subset X$ has $\inf F\in X$). Then its Dedekind topology $\Dedekind_X$ induces three weaker intrinsic topologies:
$\DWeak_X$, $\DLaw_X$ and $\DZar_X$.

The topology
\begin{itemize}
\item $\DLaw_X$ is generated by the base consisting of Dedekind open subsemilattices in $X$;
\item $\DZar_X$ is generated by the base consisting of complements to Dedekind closed subsemilattices in $X$,
\item $\DWeak_X$ is generated by the subbase consisting of the preimages $h^{-1}(U)$ of open sets $U\subset\II$ under continuous semilattice homomorphisms $h:(X,\Dedekind_X)\to\II$.
\end{itemize}
The intrinsic topology $\DZar_X$ has been considered in \cite{Law} and called the {\em lower complete topology}. The topology $\DWeak_X$ will be called {\em the $\II$-weak topology} on $X$.

A semilattice $X$ is called {\em complete} if it is complete as a poset endowed with the partial order $\le$ defined by $x\le y$ iff $xy=x$. By Lemma~\ref{l:complete-poset}, this definition of completeness is equivalent to the definition of a complete semilattice given in  \cite[O-2.1(iv)]{Bible}.

For any (complete) semilattice $X$ we have the following diagram describing the relations between the intrinsic topologies defined above. In this diagram a (dotted) arrow $\mathcal A\to\mathcal B$ indicates that $\mathcal A\subset\mathcal B$ (if the semilattice is complete).
$$\xymatrix{
\Interval_X\ar[r]&\Lawson_X\ar[r]&\DZar_X\ar[r]&\Dedekind_X\ar[r]&\Chain_X\ar@/^10pt/@{.>}[l]\\
&\Scott_X\ar[u]&\DWeak_X\ar[u]\ar@{.>}[lu]\ar[r]&\DLaw_X\ar[u]
}
$$

The following corollary of Lemma~\ref{l:c=>Zar} characterizes semilattices whose topologies $\Interval_X$, $\Lawson_X$ and $\DZar_X$ are compact, thus generalizing Theorem~III-1.9 in \cite{Bible} and completing Proposition O-2.2(iv) in \cite{Bible}. In fact, the implication $(1)\Ra(2)$ was proved in \cite[7.4]{Law}, so should be attributed to Jimmie Lawson.

\begin{theorem}[Lawson]\label{t:Law-comp} For a semilattice $X$ the following conditions are equivalent:
\begin{enumerate}
\item $X$ is complete;
\item the lower complete topology $\DZar_X$ is compact;
\item the Lawson topology $\Lawson_X$ is compact;
\item the interval topology $\Interval_X$ is compact.
\end{enumerate}
\end{theorem}

\begin{proof} The implication $(1)\Ra(2)$ follows from Lemma~\ref{l:c=>Zar} and $(2)\Ra(3)\Ra(4)$ from the inclusions $\Interval_X\subset\Lawson_X\subset\DZar_X$.
\smallskip

$(4)\Ra(1)$ Assume that the interval topology $\Interval_X$ on $X$ is compact.
To show that $X$ is a complete semilattice, we need to show that each non-empty chain $C\subset X$ has $\inf C$ and $\sup C$ in $X$.

First we show that the chain $C$ has $\inf C$. For this observe that $\F=\{{\downarrow}c:c\in C\}$ is a centered family of $\Interval_X$-closed sets in $X$. By the compactness of the interval topology $\Interval_X$, the intersection $K=\bigcap\F$ is a closed non-empty set in $(X,\Interval_X)$. By the compactness of $K$ in the interval topology, the intersection $\bigcap_{x\in K}K\cap{\uparrow}x$ is non-empty and contains the unique element $\max K$, which coincides with $\inf C$ by the definition of the greatest lower bound $\inf C$.

By analogy we can show that $C$ has $\sup(C)\in X$.
\end{proof}

A semilattice $X$ is called {\em meet continuous} if for any up-directed subset $D\subset X$ having $\sup D\in X$ and any $a\in X$ the up-directed set $aD$ has $\sup(aD)=a\cdot\sup D$. On the other hand, it is known \cite[O-1.10]{Bible} that
for any down-directed subset $D$ of a semilattice $X$ having $\inf D\in X$ and any $a\in X$ the down-directed set $aD$ automatically has $\inf(aD)=a\cdot\inf D$.

Theorems O-4.2 and III-2.8 of \cite{Bible} imply the following characterization.

\begin{theorem}\label{t:meet} For a semilattice $X$ the following conditions are equivalent:
\begin{enumerate}
\item $X$ is meet continuous;
\item $(X,\Dedekind_X)$ is a semitopological semilattice;
\item $(X,\Lawson_X)$ is a semitopological semilattice.
\item $(X,\Scott_X)$ is a semitopological semilattice.
\end{enumerate}
\end{theorem}

\begin{example} Consider the semilattice $$X=\big(\{0,1\}\times (\w\cup\{\w,\w+1\})\big)\setminus\{(1,\w)\}$$endowed with the operation of coordinatewise minimum. It is easy to see that $X$ is complete and the interval topology $\Interval_X$ is Hausdorff and coincides with the Lawson topology $\Lawson_X$, which is compact and metrizable. On the other hand, the semilattice $X$ is not meet-continuous as the chain $C=\{1\}\times\w$ has $\sup C=(1,\w+1)$ and for $a=(0,\w+1)$ the chain $aC=\{0\}\times\w$ has $\sup aC=(0,\w)\ne a\cdot\sup C$. Therefore, $(X,\Lawson_X)=(X,\Interval_X)$ is a compact Hausdorff topologized semilattice, which is not semitopological.
\end{example}

\begin{corollary}\label{c:final} For a meet-continuous complete semilattice $X$ the following conditions are equivalent:
\begin{enumerate}
\item the lower complete topology $\DZar_X$ is Hausdorff;
\item the Lawson topology $\Lawson_X$ is Hausdorff;
\item the $\II$-weak topology $\DWeak_X$ is Hausdorff;
\item $(X,\Dedekind_X)$ is a $V$-semilattice;
\item $\DWeak_X=\DZar_X$;
\item $\DWeak_X=\Lawson_X=\DZar_X$.
\end{enumerate}
\end{corollary}

\begin{proof} By Theorem~\ref{t:meet}, the topologized semilattice $(X,\Dedekind_X)$ is  a semitopological semilattice. The definition of the topology $\Dedekind_X$ implies that it satisfies the separation axiom $T_1$.

Now the equivalence of the conditions (1)--(6) follows from Theorem~\ref{t:final} and the inclusions $\DWeak_X\subset\Lawson_X\subset\DZar_X$ holding because of the completeness of $X$.
\end{proof}

\section{The interplay between weak and intrinsic topologies on  semitopological semilattices}\label{s:Scott}

In this section we investigate the interplay between the weak topologies $\Weak_X$, $\Zar_X$ and the intrinsic topologies $\DWeak_X$, $\Lawson_X$, and $\DZar_X$ on a complete semitopological semilattice $X$.

Let $X$ be a topologized semilattice endowed with the topology $\tau_X$. If the topologized semilattice $(X,\tau_X)$ is complete, then the semilattice $X$ is complete as well and hence
$$\tau_X\subset\Chain_X=\Dedekind_X,$$
which implies the inclusions $\Weak_X\subset \DWeak_X$, $\Law_X\subset\DLaw_X$ and $\Zar_X\subset\DZar_X$.

Consequently, for any complete topologized semilattice $X$ we have the following diagram in which an arrow $\mathcal A\to\mathcal B$ indicates the inclusion $\mathcal A\subset\mathcal B$.
$$
\xymatrix{
\Weak_X\ar[rr]\ar[d]&&\Zar_X\ar[d]\\
\DWeak_X\ar[r]&\Lawson_X\ar[r]&\DZar_X
}
$$

Now we find conditions on a Hausdorff semitopological semilattice $X$ guaranteeing that $\Zar_X=\DZar_X$ or even $\Weak_X=\DZar_X$.

Let us recall that a topological space $X$ satisfies the separation axiom
\begin{itemize}
\item $T_1$ if for any distinct points $x,y\in X$ there exists an open set $U\subset X$ such that $x\in U\subset X\setminus\{y\}$;
\item $T_2$ if for any distinct points $x,y\in X$ there exists an open set $U\subset X$ such that $x\in U\subset\bar U\subset X\setminus\{y\}$;
\item $T_3$ if $X$ is a $T_1$-space and for any open set $V\subset X$ and point $x\in V$ there exists an open set $U\subset X$ such that $x\in U\subset\bar U\subset V$;
\item $T_{3\frac12}$ if $X$ is a $T_1$-space and for any open set $V\subset X$ and point $x\in V$ there exists a continuous function $f:X\to[0,1]$ such that $x\in f^{-1}([0,1))\subset V$;
\item $T_{2\delta}$ if $X$ is a $T_1$-space and for any open set $U\subset X$ and point $x\in U$ there exists a countable family $\U$ of closed neighborhoods of $x$ in $X$ such that $\bigcap\U\subset U$;
\item $\vv{T}_i$ for $i\in\{1,2,2\delta,3,3\frac12\}$ if $X$ admits an injective continuous map $X\to Y$ to a $T_i$-space $Y$.
\end{itemize}
Topological spaces satisfying a separation axiom $T_i$ are called {\em $T_i$-spaces}. The separation axioms $T_{2\delta}$ and $\vv{T}_{2\delta}$ were introduced in \cite{BBR}.

The following diagram describes the implications between the separation axioms $T_i$ and $\vv{T}_i$ for $i\in\{1,2,2\delta,3,3\frac12\}$.
$$\xymatrix{
T_{3\frac12}\ar@{=>}[r]\ar@{=>}[d]&T_3\ar@{=>}[r]\ar@{=>}[d]&T_{2\delta}\ar@{=>}[r]\ar@{=>}[d]&T_2\ar@{<=>}[d]\ar@{=>}[r]&T_1\ar@{<=>}[d]\\
\vv{T}_{3\frac12}\ar@{=>}[r]&\vv{T}_3\ar@{=>}[r]&\vv{T}_{2\delta}\ar@{=>}[r]&\vv{T}_2\ar@{=>}[r]&T_1
}
$$
Observe that a topological space $X$ satisfies the separation axiom $\vv{T}_{3\frac12}$ if and only if it is {\em functionally Hausdorff\/} in the sense that for any distinct points $x,y\in X$ there exists a continuous function $f:X\to\mathbb R$ with $f(x)\ne f(y)$. Therefore, each functionally Hausdorff space is a $\vv{T}_{2\delta}$-space.

A topological space $X$ is {\em sequential} if for each non-closed subset $A\subset X$ there exists a sequence $\{a_n\}_{n\in\w}\subset A$ that converges to a point $x\in X\setminus A$.

\begin{theorem}\label{t:Law} For a complete Hausdorff semitopological semilattice $X$ the topologies $\Zar_X$ and $\DZar_X$ coincide if one of the following conditions is satisfied:
\begin{enumerate}
\item $X$ is a topological semilattice;
\item $X$ is a sequential space;
\item $X$ is a $\vec T_{2\delta}$-space;
\item $X$ is functionally Hausdorff.
\end{enumerate}
\end{theorem}

\begin{proof} Since $\Zar_X\subset\DZar_X$, it remains to establish the inclusion $\DZar_X\subset \Zar_X$. Since the topology $\DZar_X$ is generated by the subbase consisting of complements to Dedekind closed subsemilattices in $X$, it suffices to prove that each Dedekind closed subsemilattice $S$ in $X$ is closed.

The completeness of $X$ and the Dedekind closedness of $S$ in $X$ implies that the semitopological semilattice $S$ is complete and hence closed in $X$ according to the following Theorem~\ref{t:BB}.
\end{proof}

\begin{theorem}\label{t:BB} For a continuous homomorphism $h:X\to Y$ from a complete topologized semilattice $X$ to a Hausdorff semitopological semilattice $Y$ the image $h(X)$ is closed in $Y$ if one of the following conditions is satisfied:
\begin{enumerate}
\item $Y$ is a topological semilattice;
\item $Y$ is a sequential space;
\item $Y$ is a $\vec T_{2\delta}$-space;
\item $Y$ is functionally Hausdorff.
\end{enumerate}
\end{theorem}

The statements (1), (2), (3) of this theorem were proved in \cite{BBm}, \cite{BBseq} and \cite{BBR}, respectively. The statement (4) follows from (3).

\begin{problem}\label{prob:Scott} Is $\Zar_X=\DZar_X$ for any complete Hausdorff semitopological semilattice $X$?
\end{problem}

With the help of Theorem~\ref{t:Law}, we can prove the following characterization extending Theorem~\ref{t:final}.

\begin{theorem}\label{t:final2} For a complete semitopological semilattice $X$ the following conditions are equivalent:
\begin{enumerate}
\item the topology $\Zar_X$ is Hausdorff;
\item the topology $\Weak_X$ is Hausdorff;
\item $X$ is $\II$-separated;
\item $X$ is a $V$-semilattice satisfying the separation axiom $T_1$;
\item $X$ is a $T_0$-space and $\Weak_X=\Zar_X$;
\item $\Weak_X=\DZar_X$;
\item $\Weak_X=\Zar_X=\DWeak_X=\Lawson_X=\DZar_X$;
\item $X$ is functionally Hausdorff and the intrinsic topology $\DZar_X$ is Hausdorff.
\end{enumerate}
\end{theorem}

\begin{proof} The equivalences of the conditions $(1)$--$(5)$ were proved in Theorem~\ref{t:final}.
\smallskip

$(5)\Ra(6)$: Assume that $X$ is a $T_0$-space and $\Zar_X=\Weak_X$. By Theorem~\ref{t:Cc}, the topology $\Zar_X$ is compact. By Theorem~\ref{t:final},  the topology $\Zar_X$ is Hausdorff and being compact, is functionally Hausdorff. By Theorems~\ref{t:final} and \ref{t:Law},  $\Weak_X=\Zar_X=\DZar_X$.
\smallskip

$(6)\Ra(7)$: Assuimg that $\Weak_X=\DZar_X$ and taking into account that $\Weak_X\subset\Zar_X\subset\DZar_X$ and $\Weak_X\subset\DWeak_X\subset\Lawson_X\subset\DZar_X$, we conclude that $\Weak_X=\Zar_X=\DWeak_X=\Lawson_X=\DZar_X$.
\smallskip

$(7)\Ra(8)$: Assuming that $\Weak_X=\DZar_X$ and taking into account that the topology $\DZar_X$ satisfies the separation axiom $T_1$, we conclude that the weak topology $\Weak_X=\DZar_X$ is $T_1$ and (by its definition) is Tychonoff and hence functionally Hausdorff. Since $\DZar_X=\Weak_X\subset\tau_X$, the space $X$ is functionally Hausdorff.

The implication $(8)\Ra(1)$ follows from Theorem~\ref{t:Law}.
\end{proof}

\section{Acknowledgement} The authors express their sincere thanks to the referee for the very fruitful suggestion to explore the relation of the topologies $\Weak_X$, $\Zar_X$ and $\Law_X$ on a topologized semilattice $X$ to the intrinsic topologies $\Dedekind_X$, $\Scott_X$ and $\Lawson_X$ on $X$.



\begin{thebibliography}{}

\bibitem{BBm} T.~Banakh, S.~Bardyla,
\emph{Characterizing chain-finite and chain-compact topological semilattices},  Semigroup Forum {\bf 98}:2 (2019), 234--250. 

\bibitem{BBc} T.~Banakh, S.~Bardyla, {\em
Completeness and absolute $H$-closedness of topological
semilattices}, Topology Appl. {\bf 260} (2019) 189--202. 

\bibitem{BBseq} T.~Banakh, S.~Bardyla, {\em On images of complete subsemilattices in sequential  semitopological  semilattices}, Semigroup Forum, {\bf 100} (2020) 662–670. 

\bibitem{BBo} T.~Banakh, S.~Bardyla, {\em Complete topologized posets and semilattices}, Topology Proc. {\bf 57} (2021) 177--196.

\bibitem{BBR} T.~Banakh, S.~Bardyla, A.~Ravsky, {\em  The closedness of complete subsemilattices in functionally Hausdorff semitopological semilattices}, Topology Appl. {\bf 267} (2019) 106874.






\bibitem{Bardyla-Gutik-2012}
S.~Bardyla, O.~Gutik,
\emph{On $\mathscr{H}$-complete topological semilattices},
Mat. Stud. {\bf 38}:2 (2012) 118--123.


\bibitem{Bruns} G.~Bruns, {\em A lemma on directed sets and chains}, Arch. der Math. {\bf 18} (1967), 561--563.






\bibitem{CHK}
J.H.~Carruth, J.A.~Hildebrant,  R.J.~Koch, \emph{The Theory of
Topological Semigroups},
Vol. II, Marcel Dekker, Inc., New York and Basel, 1986.

\bibitem{CHR} W.W.~Comfort, K.H.~Hofmann, D.~Remus, {\em Topological groups and semigroups}, in: Recent progress in general topology (Prague, 1991),  North-Holland, Amsterdam, (1992) 57--144.

\bibitem{vD} E.K.~van Douwen, {\em The integers and topology}, in: Handbook of set-theoretic topology, North-Holland, Amsterdam, (1984) 111--167,

%
\bibitem{Engelking1989}
R.~Engelking, \emph{General Topology}, Heldermann,
Berlin, 1989.






\bibitem{Bible}
G.~Gierz, K.H.~Hofmann, K.~Keimel, J.D.~Lawson,
M.W.~Mislove, D.S.~Scott, \emph{Continuous Lattices and Domains}.
Cambridge Univ. Press, Cambridge, 2003.








\bibitem{GutikRepovs2008}
O.~Gutik, D.~Repov\v{s}, {\em On linearly ordered $H$-closed
topological semilattices}, Semigroup Forum \textbf{77}:3 (2008),
474--481.


\bibitem{Iwamura}  T. Iwamura, {\em A lemma on directed sets}, Zenkoku Shijo Sugaku Danwakai {\bf 262} (1944), 107--111 (in Japanese).


\bibitem{Law} J.~Lawson, {\em Intrinsic lattice and semilattice topologies},  Proceedings of the University of Houston Lattice Theory Conference (Houston, Tex., 1973), pp. 206--260.

\bibitem{Law69} J.D.~Lawson, {\em Topological semilattices with small semilattices}, J. London Math., \textbf{2} (1969), 719--724.

\bibitem{Law74} J.D.~Lawson, {\em Joint continuity in semitopological semigroups}, Illinois J. Math. \textbf{18}:2 (1974), 275--285.



\bibitem{Mark} G.~Markowsky, {\em Chain-complete posets and directed sets with applications}, Algebra Universalis, {\bf 6} (1976) 53-68.





\bibitem{Stepp1969}
J.W.~Stepp, {\it A note on maximal locally compact semigroups}.
Proc. Amer. Math. Soc. {\bf 20}  (1969), 251---253.

\bibitem{Stepp1975}
J.W.~Stepp, {\it Algebraic maximal semilattices}. Pacific J. Math.
{\bf 58}:1  (1975), 243---248.


\bibitem{Vai} R.~Vainio, {\em A maximal chain approach to topology and orde},  Internat. J. Math. Math. Sci. {\bf 11}:3 (1988), 465--472.


\end{thebibliography}
\end{document}